\documentclass[11pt]{amsart}

\usepackage{amsmath, amscd, amssymb}
\usepackage{latexsym, amssymb, amsmath, amscd, amsfonts}
\usepackage{youngtab}  
\usepackage{url}
\usepackage{tikz}
\usepackage{listings} 
\usepackage{color}

\numberwithin{equation}{section}
\newtheorem{theorem}{Theorem}[section]

\newtheorem{definition}[theorem]{Definition}

\newtheorem{lemma}[theorem]{Lemma}
\newtheorem{proposition}[theorem]{Proposition}

\newtheorem{notation}[theorem]{Notation}
\begin{document}

\title[Hive Algebras and Tensor Product Algebras]
{On the structures of hive algebras and tensor product algebras 
for general linear groups of low rank}

\author{Donggyun Kim}
\email{kim.donggyun@gmail.com}
\address{Department of Mathematics, Korea University, 
145 Anam-ro Seongbuk-gu, Seoul 02841, South Korea}

\author{Sangjib Kim}
\email[Corresponding author]{sk23@korea.ac.kr}
\address{Department of Mathematics, Korea University, 
145 Anam-ro Seongbuk-gu, Seoul 02841, South Korea}

\author{Euisung Park}
\email{euisungpark@korea.ac.kr}
\address{Department of Mathematics, Korea University, 
145 Anam-ro Seongbuk-gu, Seoul 02841, South Korea}

\begin{abstract}
The tensor product algebra $\mathrm{TA}(n)$ 
for the complex general linear group $\mathrm{GL}(n)$,
introduced by Howe et al., describes 
the decomposition of tensor products of irreducible polynomial 
representations of $\mathrm{GL}(n)$. Using the hive model for 
the Littlewood-Richardson coefficients, we provide a finite presentation of 
the algebra $\mathrm{TA}(n)$ for $n=2, 3, 4$ 
in terms of generators and relations, thereby giving a description of 
highest weight vectors of irreducible representations in the tensor products. 
We also compute the generating function of certain 
sums of Littlewood-Richardson coefficients.
\end{abstract}

\subjclass[2010]{20G05, 13A50, 05E15}
\keywords{General linear group, Highest weight vector, 
Littlewood-Richardson coefficients, Tensor product decomposition, 
Tensor product algebra, Hive, Hilbert-Poincare series.}

\maketitle

%%%%%%%%%%%%%%%%%%%%%%%%%%%%%%%%%%%%%%%%%%%%%%%%%%%%%%%%%%%%%%%%%%%%%%%%%%%%%%
%%%%%%%%%%%%%%%%%%%%%%%%%%%%%%%%%%%%%%%%%%%%%%%%%%%%%%%%%%%%%%%%%%%%%%%%%%%%%%

\section{Introduction}

%%%%%%%%%%%%%%%%%%%%%%%%%%%%%%%%%%%%%%%%%%%%%%%%%%%%%%%%%%%%%%%%%%%%%%%%%%%%%%

For the complex general linear group $\mathrm{GL}(n)=\mathrm{GL}_n(\mathbb{C})$, 
the group of $n\times n$ invertible matrices over the field $\mathbb{C}$ 
of complex numbers, we let $V^{\lambda}_n$ denote its polynomial irreducible 
representation labeled by a Young diagram $\lambda$.
The \emph{Littlewood-Richardson}(LR) \emph{coefficient} 
for $\mathrm{GL}(n)$ is the multiplicity $c_{\mu\nu}^{\lambda}$ of 
the irreducible representation $V^{\lambda}_n$ occurring 
in the decomposition of the tensor product of two irreducible 
representations $V_n^{\mu}$ and $V_n^{\nu}$
\begin{equation}\label{tensor-decomp}
V_n^{\mu} \otimes V_n^{\nu} \; = \; \bigoplus_{\lambda} 
      \left( V^{\lambda}_n \right)^{\oplus \, c^{\lambda}_{\mu \nu}}.
\end{equation}
Since the Schur polynomials in $n$ variables are the characters of 
polynomial irreducible representations of $\mathrm{GL}(n)$, the LR coefficients 
can be also defined from the decomposition of the product of two Schur polynomials
\[
s^{\mu}_{n} \, s^{\nu}_{n} \; = \; \sum_{\lambda} c^{\lambda}_{\mu\nu} \, s^{\lambda}_{n}.
\]  

The LR coefficients are usually described in terms of 
a family of combinatorial objects such as LR tableaux. 
For various combinatorial objects counting the LR coefficients, 
see, for example, \cite{DK16, PV05, Pu08, TY08, vL01}.
A very different approach was proposed by Howe et al. in 
\cite{HTW05, HJLTW, HL12}. Using classical invariant theory, they constructed 
a multi-graded algebra, which we call \textit{the $\mathrm{GL}(n)$ tensor product algebra} 
and denote by $\mathrm{TA}(n)$. The dimension of its $(\lambda, \mu, \nu)$-homogeneous 
component $\mathrm{TA}(n)^{\lambda}_{\mu\nu}$ is exactly the LR coefficient 
$c_{\mu \nu}^{\lambda}$, and for each space $\mathrm{TA}(n)^{\lambda}_{\mu\nu}$
they gave a description of $\mathbb{C}$-basis elements labeled 
by LR tableaux on $\lambda/\mu$ with content $\nu$ (see \cite[\S 2]{HTW05}).
From a representation theoretic point of view, this approach provides 
a lot more refined information 
than many combinatorial ones in that the space 
$\mathrm{TA}(n)^{\lambda}_{\mu\nu}$ consists of the highest weight vectors of the
isomorphic copies of $V^{\lambda}_{n}$ occurring in the decomposition 
of $V_n^{\mu} \otimes V_n^{\nu}$.
From the existence of a finite SAGBI basis, it is also shown 
 that the algebra $\mathrm{TA}(n)$ is a flat deformation of its initial algebra 
(see \cite[\S 2.4-2.5]{HJLTW}). 

\medskip

In this paper, for each  $n=2, 3, 4$, we give a finite presentation of 
the $\mathrm{GL}(n)$ tensor product algebra $\mathrm{TA}(n)$ in terms of generators and relations. 
This will provide a method, different from the one given in \cite{HTW05}, to construct 
highest weight vectors appearing in the decomposition of the tensor products  \eqref{tensor-decomp}.
An explicit example 
illustrating how to compute highest weight vectors corresponding to hives 
or LR tableaux is given in \S \ref{ex33}.
Also, by applying a technique in projective algebraic geometry to 
the hive model for the LR coefficients, we compute 
a closed form formula for the series
\[
\mathfrak{HP}_n(t) = \sum_{d \geq 0}  m_d \, t^d
\]
where $m_d$ is the sum of LR coefficients $c_{\mu \nu}^{\lambda}$ 
over the triples $(\lambda,\mu,\nu)$ of Young diagrams 
such that $d=|\lambda|=|\mu|+|\nu|$.

\medskip

We remark that for the special linear group $G=\mathrm{SL}_n(\mathbb{C})$ 
and its standard maximal unipotent subgroup $U$, Grosshans studied 
in \cite{Gr} the algebra of the invariants of $G$ acting on 
$G/U \times G/U \times G/U$ by left translation. In particular, for $n=2,3,4$, 
he gave the invariants and all relations among them, and showed how 
such results are related to the tensor product decomposition problem 
for the representations of $G$. See also \cite[\S 4]{HTW05} 
and references therein for the related works on the case $n=4$.

\bigskip

%%%%%%%%%%%%%%%%%%%%%%%%%%%%%%%%%%%%%%%%%%%%%%%%%%%%%%%%%%%%%%%%%%%%%%%%%%%%%%
%%%%%%%%%%%%%%%%%%%%%%%%%%%%%%%%%%%%%%%%%%%%%%%%%%%%%%%%%%%%%%%%%%%%%%%%%%%%%%

\section{Hive Algebra and its HP series}\label{hive-algebra}

%%%%%%%%%%%%%%%%%%%%%%%%%%%%%%%%%%%%%%%%%%%%%%%%%%%%%%%%%%%%%%%%%%%%%%%%%%%%%%

In this section, we impose a monoid structure on the collection of 
hives for $\mathrm{GL}(n)$, investigate the structure of 
the associated monoid algebra, and then
compute its Hilbert-Poincare(HP) series. These results will be extended to 
the $\mathrm{GL}(n)$ tensor product algebras in \S \ref{tensor-algebra}.

%%%%%%%%%%%%%%%%%%%%%%%%%%%%%%%%%%%%%%%%%%%%%%%%%%%%%%%%%%%
\subsection{LR tableaux}

First let us recall that a LR tableau $T$ for $\mathrm{GL}(n)$ is 
a filling of skew Young diagram with $1,2,...,n$ satisfying 
the semistandard condition and the Yamanouchi word, or reverse lattice word, condition.
For example, the following tableaux 
\begin{equation*}
\young(\ \ \ \ \ 111111,\ \ \ 1222,\ 2333,244)
\end{equation*}
is a LR tableau for $\mathrm{GL}(4)$  on a skew Young 
diagram $(11,7,5,3)/(5,3,1,0)$ with content $(7, 5, 3, 2)$.
It is well known that for Young diagrams $\lambda$, $\mu$, 
and $\nu$ having at most $n$ rows, the number of LR tableaux on 
the skew Young diagram $\lambda/\mu$ with content $\nu$ is 
equal to the LR coefficient $c^{\lambda}_{\mu \nu}$ for $\mathrm{GL}(n)$. 
See, for example, \cite{Fu97, HL12, Ma95, St99, vL01}.

%%%%%%%%%%%%%%%%%%%%%%%%%%%%%%%%%%%%%%%%%%%%%%%%%%%%%%%%%%%
\subsection{Hives}

Now we recall the hive model for the LR coefficients 
introduced by Knutson and Tao \cite{KT99}.
A sequence $\kappa=(\kappa_1, \kappa_2, ..., \kappa_n) \in \mathbb{Z}^n$ 
is \textit{non-negative dominant}, if 
\[
k_1 \geq k_2 \geq \cdots \geq k_n \geq 0.
\]
We write $|\kappa|$ for the sum $\kappa_1 + \cdots + \kappa_n$. 
Note that we can identify a non-negative dominant sequence $\kappa$
with the Young diagram having $\kappa_i$ boxes in its $i$-th row  
counting from top to bottom.

A  \textit{triangular array} $h=(h_{ij})_{1 \leq j \leq i \leq n+1}$ is 
an array of integers  whose $i$-th row contains a subsequence $(h_{i1}, 
h_{i2}, ..., h_{ii})$ of length $i$ for $1 \leq i \leq n+1$. For example, 
if $n=4$ then $h=(h_{ij})$ can be drawn as in Figure \ref{triang-gl4}.
\begin{figure}[!ht]
\begin{tikzpicture}[scale=0.80] 
\node at (5,5) {$h_{11}$} ;
\node at (4,4) {$h_{21}$} ; \node at (6,4) {$h_{22}$} ;
\node at (3,3) {$h_{31}$} ; \node at (5,3) {$h_{32}$} ; \node at (7,3) {$h_{33}$} ; 
\node at (2,2) {$h_{41}$} ; \node at (4,2) {$h_{42}$} ; \node at (6,2) {$h_{43}$} ; \node at (8,2) {$h_{44}$} ;
\node at (1,1) {$h_{51}$} ; \node at (3,1) {$h_{52}$} ; 
\node at (5,1) {$h_{53}$} ; \node at (7,1) {$h_{54}$} ; \node at (9,1) {$h_{55}$} ; 
\end{tikzpicture}
\centering
\caption{The triangular array for $\mathrm{GL}(4)$}\label{triang-gl4}
\end{figure}

The \textit{boundary} of a triangular array $h$ is a triple
$(\underline{a}_h, \underline{b}_h,\underline{c}_h)$ where
\begin{align*} 
\underline{a}_h &=  (a_1, a_2, ..., a_n) \text{\ where \ }
a_i = h_{i+1,i+1}- h_{i,i}, \\
\underline{b}_h &= (b_1, b_2, ..., b_n) \text{\ where \ } b_i = h_{i+1,1}-h_{i,1}, \\
\underline{c}_h &=  (c_1, c_2, ..., c_n) \text{\ where \ } c_i = h_{n+1, i+1} - h_{n+1, i}.
\end{align*}
We say the boundary of $h$ is \textit{non-negative dominant}, if 
the sequences $\underline{a}_h$, $\underline{b}_h$, and $\underline{c}_h$ 
in the boundary of $h$ are non-negative dominant. 
The following inequalities for a triangular array  
$h=(h_{ij})_{1 \leq j \leq i \leq n+1}$ are called the \textit{rhombus conditions}:
\begin{itemize}
\item  $(h_{i,j}+h_{i+1,j+1}) \geq (h_{i+1,j}+h_{i,j+1})$ 
for $1 \leq j < i \leq n$,
\item  $(h_{i,j}+h_{i,j+1}) \geq (h_{i+1,j+1}+h_{i-1,j})$ 
for $1 \leq j < i \leq n$, 
\item $(h_{i+1,j}+h_{i,j}) \geq (h_{i+1,j+1} + h_{i,j-1})$ 
for $1 < j \leq i \leq n$.
\end{itemize}
In other words, 
for each fundamental rhombus in a triangular array,
the sum of entries at the obtuse corners is bigger than or equal 
to the sum of entries at the acute corners.

\begin{definition}\label{definition-hive}
A \textit{hive} for $\mathrm{GL}(n)$ is 
a triangular array $h=(h_{ij})_{1 \leq j \leq i \leq n+1}$ of  
integers such that
\begin{enumerate}
\item $h_{11}=0$;
\item the entries $h_{ij}$ satisfy all three types of rhombus conditions;
\item the boundary of $h$ is non-negative dominant.
\end{enumerate}
\end{definition}

For Young diagrams $\lambda, \mu, \nu$ with not more than $n$ rows, 
it is well known that the number of hives with boundary $(\lambda,\mu,\nu)$ 
is exactly the LR coefficient $c_{\mu \nu}^{\lambda}$ for $\mathrm{GL}(n)$. 
See, for example, \cite{Bu00, KT99, KTW}. 
Therefore, we can expect a one-to-one correspondence between 
the set of all LR tableaux on $\lambda/\mu$ with content $\nu$ and 
the set of all hives with boundary $(\lambda,\mu,\nu)$. 
The following is from \cite[\S 3]{KTT}. See also \cite{Bu00, PV05} and \cite[\S 4.5]{DK16}.
\begin{lemma}
\label{LR-Hive-bij}
The following map gives a one-to-one correspondence between 
the set of all LR tableaux $T$ for $\mathrm{GL}(n)$ on $\lambda/\mu$ with content $\nu$ and 
the set of all hives $h_T$ for $\mathrm{GL}(n)$  with boundary $(\lambda,\mu,\nu)$:
\[
T \mapsto h_T= (h_{ij})_{1 \leq j \leq i \leq n+1}
\]
where $h_{ij}$ is the number of empty boxes and boxes containing entries 
not more than $j-1$ in the first $i-1$ rows of $T$.
\end{lemma}

We remark that if $|\lambda| \ne |\mu|+|\nu|$ then $c_{\mu \nu}^{\lambda} =0$; 
if $(h_{ij})_{1 \leq j \leq i \leq n+1}$ is a hive for $\mathrm{GL}(n)$ 
whose boundary is $(\lambda, \mu, \nu)$ then we have
\begin{equation} \label{Z-grading}
h_{n+1,n+1} = |\lambda|=|\mu|+|\nu|.
\end{equation}

%%%%%%%%%%%%%%%%%%%%%%%%%%%%%%%%%%%%%%%%%%%%%%%%%%%%%%%%%%%%%%%%%%%%%%%%%%%%%%

\subsection{Hive algebra and HP series} 

Let us impose a monoid structure on the set of hives for $\mathrm{GL}(n)$.
Note that for two hives $h=(h_{ij})$ and $h'=(h'_{ij})$ for $\mathrm{GL}(n)$, 
as elements of the free abelian group $\mathbb{Z}^{(n+1)(n+2)/2}$, 
the entries of their sum
\[
h + h' =(h_{ij}+h'_{ij})_{1 \leq j \leq i \leq n+1}
\] 
also satisfy the rhombus conditions, and the boundary of their sum
\[
(\underline{a}_h + \underline{a}_{h'}, \underline{b}_h + \underline{b}_{h'}, 
\underline{c}_h + \underline{c}_{h'})
\]
is also non-negative dominant.

\begin{definition}[Hive cone and Hive algebra] 
\ 
\begin{enumerate}
\item 
The \textit{hive cone} $\mathcal{H}(n)$  for $\mathrm{GL}(n)$ is the submonoid 
of $\mathbb{Z}^{(n+1)(n+2)/2}$ consisting of all the hives for $\mathrm{GL}(n)$.

\item 
The \textit{hive algebra} $\mathrm{HA}(n)$ for $\mathrm{GL}(n)$ is 
the subalgebra of the polynomial algebra 
\[
\mathrm{HA}(n) \subset \mathbb{C}[z_{ij}: 1 \leq j \leq i \leq n+1]
\]
over $\mathbb{C}$ generated by the monomials
$\mathbf{z}^h=\prod_{i,j}z_{ij}^{h_{ij}}$ for all  $h=(h_{ij}) \in \mathcal{H}(n)$.
\end{enumerate}
\end{definition}

From $\mathbf{z}^h \cdot \mathbf{z}^{h'} = \mathbf{z}^{h+h'}$, it follows
that the hive algebra $\mathrm{HA}(n)$ is isomorphic to the monoid algebra of 
the hive cone $\mathcal{H}(n)$. 
With the subspace $\mathrm{HA}(n)_{\mu \nu}^{\lambda}$ of $\mathrm{HA}(n)$ 
spanned by the monomials $\mathbf{z}^h$ 
corresponding to all the hives $h$ with boundary $(\lambda,\mu,\nu)$, 
the hive algebra $\mathrm{HA}(n)$ is multi-graded by the triples 
$(\lambda,\mu,\nu)$ of non-negative dominant sequences. 

Using \eqref{Z-grading}, we can also consider the $\mathbb{Z}$-grading 
structure of the hive algebra $\mathrm{HA}(n)$  with respect to 
the degree of $z_{n+1,n+1}$. That is, 
\begin{equation} \label{degree-d}
\mathrm{HA}(n) 
= \bigoplus_{(\lambda,\mu,\nu)} \mathrm{HA}(n)_{\mu \nu}^{\lambda} 
= \bigoplus_{d \geq 0} \bigoplus_{\stackrel{(\lambda,\mu,\nu):}{|\lambda|=d}} 
  \mathrm{HA}(n)_{\mu \nu}^{\lambda}.
\end{equation}
With this $\mathbb{Z}$-grading, we can consider the HP 
series $\mathfrak{HP}_n (t)$ of the hive algebra 
$\mathrm{HA}(n)$
\[
\mathfrak{HP}_n(t) = \sum_{d \geq 0}  m_d \, t^d
\]
where $m_d$ is the dimension of the $d$-homogeneous space of $\mathrm{HA}(n)$,
or equivalently, the number of hives with boundary $(\lambda,\mu,\nu)$ such that 
$|\lambda| = |\mu| + |\nu| =d$.
Since the number of hives with boundary $(\lambda,\mu,\nu)$ is equal to 
the LR coefficient $c_{\mu\nu}^{\lambda}$ and 
if $|\lambda| \ne |\mu|+|\nu|$ then $c_{\mu \nu}^{\lambda} =0$, we have
 
\begin{proposition} \label{md-LRsum}
The coefficient $m_d$ of the HP series $\mathfrak{HP}_{n}(t)$  
of $\mathrm{HA}(n)$ is the sum 
\[
m_d =  \sum_{\stackrel{(\lambda,\mu,\nu):}{|\lambda|=d}} c_{\mu \nu}^{\lambda}
\]
of the LR coefficients $c_{\mu \nu}^{\lambda}$  for $\mathrm{GL}(n)$ 
over $(\lambda,\mu,\nu)$ with $d=|\lambda|=|\mu|+|\nu|$.
\end{proposition}

%%%%%%%%%%%%%%%%%%%%%%%%%%%%%%%%%%%%%%%%%%%%%%%%%%%%%%%%%%%%%%%%%%%%%%%%%%%%%%
\medskip

Now for each $n=2, 3, 4$, we give a finite presentation of 
the hive algebra  $\mathrm{HA}(n)$ and 
compute its HP series. 

\begin{theorem}[Hive algebra for $\mathrm{GL}(2)$] \label{hive-alg-gl2} \ 
\begin{enumerate}
\item The hive algebra $\mathrm{HA}(2)$ is isomorphic to the polynomial 
	  algebra in five indeterminates.
\[
\mathrm{HA}(2) \cong \mathbb{C}[x_1, x_2,..., x_5].
\]
\item The HP series $\mathfrak{HP}_2 (t)$ of the hive algebra $\mathrm{HA}(2)$ is
\begin{align*}
\mathfrak{HP}_2 (t) &= \frac{1}{(1-t)^2 (1-t^2)^3} \\
   &=1 + 2 t + 6 t^2 + 10 t^3 + 20 t^4 + 30 t^5 + 50 t^6 + 70 t^7  + 105 t^8 + 140 t^9 + \cdots .
\end{align*}
\end{enumerate}
\end{theorem}
\begin{proof}
For Statement (1), by a direct computation or by a software tool for 
computing lattice points, for example Normaliz \cite{BIRS}, 
we can easily obtain the following Hilbert basis of
the hive cone $\mathcal{H}(2)$.
\begin{align*}
& h_1 =\left[
\begin{array}{ccccc}
  &   & 0 &   & \\
  & 0 &   & 1 &  \\
0 &   & 1 &   & 2 
\end{array}
\right], \ 
h_2 = \left[
\begin{array}{ccccc}
  &   & 0 &   & \\
  & 1 &   & 1 &  \\
1 &   & 2 &   & 2 
\end{array}
\right], \ 
h_3 =\left[
\begin{array}{ccccc}
  &   & 0 &   & \\
  & 1 &   & 1 &  \\
1 &   & 1 &   & 1 
\end{array}
\right], \ \\
& h_4 = \left[
\begin{array}{ccccc}
  &   & 0 &   & \\
  & 1 &   & 1 &  \\
2 &   & 2 &   & 2 
\end{array}
\right], \ 
h_5 = \left[
\begin{array}{ccccc}
  &   & 0 &   & \\
  & 0 &   & 1 &  \\
0 &   & 1 &   & 1 
\end{array}
\right].    
\end{align*}
The monomials $\mathbf{z}^{h_i} \in \mathrm{HA}(2)$ corresponding 
to the generators $h_i$ of $\mathcal{H}(2)$ 
\[
z_{22}z_{32}z_{33}^2, \ \ z_{21}z_{22}z_{31}z_{32}^2 z_{33}^2, \ \ z_{21}z_{22}z_{31}z_{32}z_{33}, 
\ \ z_{21}z_{22}z_{31}^2 z_{32}^2 z_{33}^2, \ \  z_{22}z_{32}z_{33}
\]
generate the algebra $\mathrm{HA}(2)$.
Since these generators are algebraically independent,  
$\mathrm{HA}(2)$ is isomorphic to the polynomial algebra in five indeterminates.

For Statement (2), with the grading \eqref{degree-d}, note that the degrees of the generators are
\[
\deg(\mathbf{z}^{h_1})=2, \ \deg(\mathbf{z}^{h_2})=2, \ \deg(\mathbf{z}^{h_3})=1, \ \deg(\mathbf{z}^{h_4})=2, \ \deg(\mathbf{z}^{h_5})=1.
\]
Then, the hive algebra $\mathrm{HA}(2)$
is isomorphic to the weighted polynomial ring corresponding 
to the weighted projective space $\mathbb{P}(2, 2, 1, 2, 1)$.
Every monomial in $\mathbb{C}[x_1, x_2,..., x_5]$ appears in
\[
\prod_{i=1}^5 \frac{1}{1-x_i} 
= \sum_{(m_1, ..., m_5)} x_1^{m_1}x_2^{m_2} \cdots x_5^{m_5},
\]
and therefore, by replacing $x_i$ with $t^{d_{i}}$ where 
$d_i =  \deg(\mathbf{z}^{h_i})$ for $1 \leq i \leq 5$, 
we obtain the HP series of $\mathrm{HA}(2)$.
\end{proof}

We note that there are ten LR tableaux for $\mathrm{GL}(2)$ 
on skew diagrams whose outer diagrams have three boxes.
\begin{align}\label{LR-tab-gl2}
&\young(\ \ \ ), \ \ \young(\ \ ,\ ), \ \ \young(\ \ 1),
  \ \  \young(\ \ ,1), \ \  \young(\ 1,\ ),\\ 
&\young(\ 11), \ \ \young(\ 1,1), \ \  \young(\ 1,2), 
  \ \ \young(111), \ \ \young(11,2). \notag
\end{align}
As we observed in Proposition \ref{md-LRsum}, this agrees with 
the coefficient of $t^3$ in the above series $\mathfrak{HP}_2 (t)$.

%%%%%%%%%%%%%%%%%%%%%%%%%%%%%%%%%%%%%%%%%%%%%%%%%%%%%%%%%%%%%%%%%%%%%%%%%%%%%%

\begin{theorem}[Hive algebra for $\mathrm{GL}(3)$] \ \label{Hive-alg-GL3}
\begin{enumerate}
\item The hive algebra $\mathrm{HA}(3)$ is isomorphic to the quotient
of the polynomial algebra in ten indeterminates $x_i$ 
by the ideal generated by $ x_1 x_ 6 x_7 - x_5 x_{10}$.
\[
\mathrm{HA}(3) \cong \mathbb{C}[x_1, ..., x_{10}] / \langle x_1 x_ 6 x_7 - x_5 x_{10}\rangle.
\]

\item The HP series $\mathfrak{HP}_3 (t)$ of $\mathrm{HA}(3)$ is 
\begin{align*}
\mathfrak{HP}_3 (t) 
&= \frac{1- t^6}{(1-t)^2 (1-t^2)^3 (1-t^3)^4 (1-t^4)} \\
&=1 + 2 t + 6 t^2 + 14 t^3 + 29 t^4 + 56 t^5 + 
 105 t^6 + 182 t^7 + 308 t^8  + 502 t^9 + \cdots .
\end{align*} 
\end{enumerate}
\end{theorem}

\begin{proof}
For Statement (1), first we 
recall that the hive cone $\mathcal{H}(3)$ is the collection 
of all the integral points
\[
(h_{11}, h_{21}, h_{22}, h_{31}, h_{32}, h_{33}, h_{41}, h_{42}, h_{43}, h_{44}) 
\]
satisfying the conditions in Definition \ref{definition-hive}.
Using a computer program 
such as Normaliz \cite{BIRS}, it is straightforward to verify that 
the following ten elements form 
the Hilbert basis of  the hive cone $\mathcal{H}(3)$:
\begin{align*}
&h_1 = (0, 0, 1, 0, 1, 1, 0, 1,1,1),
&h_2 =(0, 0, 1, 0, 1, 2, 0, 1, 2, 3),\\
&h_3 = (0, 0, 1, 0, 1, 2, 0, 1 , 2, 2),
&h_4 = (0, 1, 1, 1, 1, 1, 1, 1, 1, 1), \\
&h_5 =(0, 1, 1, 1, 2, 2, 1, 2,2,2),
&h_6 =(0, 1, 1, 1, 2, 2, 1, 2, 3,3), \\
&h_7 =(0, 1, 1, 2, 2, 2, 2, 2,2,2), 
&h_8 =(0, 1, 1, 2, 2, 2, 2, 3,3,3), \\
&h_9 =(0, 1, 1, 2, 2, 2, 3,3,3,3), 
&h_{10} =(0, 1, 2, 2, 3,3,2,3,4,4).
\end{align*}

On the other hand, by computing the kernel 
of the map from the polynomial algebra $\mathbb{C}[x_1,...,x_{10}]$ to 
the hive algebra $\mathrm{HA}(3)$ sending $x_i$ to $\mathbf{z}^{h_i}$ 
(by using a computer program such as Macaulay 2 \cite{GS}), we 
find the following basic relation:
\begin{equation*} 
h_1 + h_6 + h_7 = h_5 + h_{10} \text{\ \ and therefore\ \ } 
\mathbf{z}^{h_1} \cdot \mathbf{z}^{h_6} \cdot \mathbf{z}^{h_7} = 
\mathbf{z}^{h_5} \cdot \mathbf{z}^{h_{10}}
\end{equation*}
in $\mathcal{H}(3)$ and $\mathrm{HA}(3)$ respectively. 

For Statement (2), note that since the $\mathbb{Z}$-grading 
of the algebra $\mathrm{HA}(3)$ is given by the degree of $z_{44}$ (see \eqref{degree-d}),  
the degrees of our generators are
\begin{align*}
& \deg(\mathbf{z}^{h_1})=1, \ \deg(\mathbf{z}^{h_2})=3, \ \deg(\mathbf{z}^{h_3})=2,
 \ \deg(\mathbf{z}^{h_4})=1, \ \deg(\mathbf{z}^{h_5})=2, \\
& \deg(\mathbf{z}^{h_6})=3, \ \deg(\mathbf{z}^{h_7})=2, \ \deg(\mathbf{z}^{h_8})=3, 
\ \deg(\mathbf{z}^{h_9})=3, \ \deg(\mathbf{z}^{h_{10}})=4.
\end{align*}
Then, the hive algebra $\mathrm{HA}(3)$ is isomorphic to 
the ring of a hypersurface determined by the homogeneous polynomial  
$f=x_1 x_6 x_7 - x_5 x_{10}$ in the weighted projective space
\[ 
\mathbb{P}(1, 3, 2, 1, 2, 3, 2, 3, 3, 4).
\] 
Using the same argument given in the proof of Theorem \ref{hive-alg-gl2} (2),
we can compute the HP series of the above weighted projective space, which is
\[
H(t)=\frac{1}{(1-t)^2 (1-t^2)^3 (1-t^3)^4 (1-t^4)}.
\]
Next, we note that the degree of $f$ is $6$. 
By comparing the space of the degree $d$ elements in $\mathbb{C}[x_1, ..., x_{10}]$ and
the space of the degree $d$ elements of the form $f\cdot g$ for $g \in 
\mathbb{C}[x_1, ..., x_{10}]$, we obtain the HP series of the quotient 
$\mathbb{C}[x_1, ..., x_{10}] / \langle f \rangle$  
\[
(1 - t^6) \cdot H(t),
\]
which is the HP series of $\mathrm{HA}(3)$ in the statement.
One can also use some general results on the HP series of a graded ring 
corresponding to a complete intersection 
in a weighted projective space. See for example \cite[\S 3.4]{D82}.
\end{proof}

There are fourteen LR tableaux for $\mathrm{GL}(3)$ 
on skew diagrams whose outer diagrams have three boxes.
\begin{align}\label{LR-tab-gl3}
&\young(\ \ \ ), \ \ \young(\ \ ,\ ), \ \ \young(\ ,\ ,\ ), \ \ \young(\ \ 1),
\ \  \young(\ \ ,1), \ \  \young(\ 1,\ ), \ \ \young(\ ,\ ,1),\\ 
&\young(\ 11), \ \ \young(\ 1,1), \ \  \young(\ 1,2), \ \ \young(\ ,1,2), 
\ \ \young(111), \ \ \young (11,2), \ \ \young(1,2,3). \notag
\end{align}
As in Proposition \ref{md-LRsum}, this agrees with 
the coefficient of $t^3$ in $\mathfrak{HP}_3 (t)$.

%%%%%%%%%%%%%%%%%%%%%%%%%%%%%%%%%%%%%%%%%%%%%%%%%%%%%%%%%%%%%%%%%%%%%%%%%%%%%%

\begin{theorem}[Hive algebra for $\mathrm{GL}(4)$] \  \label{Hive-alg-4}
\begin{enumerate}
\item The hive algebra $\mathrm{HA}(4)$ is isomorphic to the quotient
of the polynomial algebra $\mathbb{C}[x_1, ..., x_{20}]$ 
by the ideal generated by the following fifteen basic relations
\begin{align*}
&r_1: x_1  x_7 x_9 - x_6  x_{15}, 
&r_2: x_1  x_8  x_9 - x_6 x_{16}, \\
&r_3: x_1 x_{11} x_{12} - x_{10} x_{18}, 
&r_4: x_6  x_{11}  x_{12} - x_{10} x_{20}, \\
&r_5: x_2  x_8  x_{10} - x_7 x_{17}, 
&r_6: x_2  x_8  x_{12} - x_7 x_{19}, \\
&r_7: x_6  x_{18} - x_1 x_{20}, 
&r_8: x_7 x_{16} - x_8 x_{15}, \\
&r_9: x_{10}  x_{19} - x_{12} x_{17},  
&r_{10}: x_{15} x_{17} - x_{2}  x_{10}  x_{16}, \\
&r_{11}: x_{15} x_{19} - x_{2} x _{12} x_{16}, 
&r_{12}: x_{15} x_{20} - x_{7} x_{9} x_{18}, \\
&r_{13}: x_{16} x_{20} - x_8 x_{9} x_{18}, 
&r_{14}: x_{17} x_{20} - x_{6} x_{11} x_{19},\\
&r_{15}: x_{17} x_{18} - x_{1} x_{11} x_{19}.
\end{align*}

\item The HP series of $\mathrm{HA}(4)$ is 
$\mathfrak{HP}_4 (t) = f(t)/g(t)$ where
\begin{align*}
f(t) &=1 -2 t -2 t^2 + 10 t^3 -2 t^4 -24 t^5 + 22 t^6 + 32 t^7 -54 t^8 \\
     & \quad -18 t^9 + 80 t^{10} -14 t^{11} -72 t^{12} + 34 t^{13} 
 + 44 t^{14} -18 t^{15} -25 t^{16} \\ 
     & \quad -18 t^{17} + 44 t^{18} + 34 t^{19} -72 t^{20} -14 t^{21} + 
       80 t^{22} -18 t^{23} -54 t^{24} \\
     & \quad + 32 t^{25} + 22 t^{26} -24 t^{27} -2 t^{28} + 10 t^{29} -2 t^{30} 
 -2 t^{31} + t^{32}; \\
g(t) &= (1-t)^4 (1-t^2)^6 (1-t^{12})^4,
\end{align*}
which is
\[
\mathfrak{HP}_4 (t) =1 + 2 t + 6 t^2 + 14 t^3 + 34 t^4 + 68 t^5 + 
 142 t^6 + 268 t^7 + 508 t^8 + 902 t^9 + \cdots. 
\]
\end{enumerate}
\end{theorem}

We will verify these statements with the aid of a computer. 
The Normaliz  codes we used 
are given in \S \ref{appendix}.
First, let us recall that the hive monoid $\mathcal{H}(4)$ 
for $\mathrm{GL}(4)$ is the collection of all the
triangular arrays of integers in Figure \ref{triang-gl4}
or the integral points
\[
(h_{11}, h_{21}, h_{22}, h_{31}, h_{32}, h_{33}, h_{41}, h_{42}, h_{43}, h_{44}, 
h_{51}, h_{52}, h_{53}, h_{54}, h_{55})
\]
satisfying the conditions in Definition \ref{definition-hive}.
Then, we can compute the Hilbert basis of the hive cone $\mathcal{H}(4)$ 
by using Normaliz \cite{BIRS}:
\begin{align*}
&h_1 = (0, 0, 1, 0, 1, 1, 0, 1,1,1, 0, 1,1,1,1),
&h_2 =(0, 0, 1, 0, 1, 2, 0, 1, 2, 2, 0, 1, 2, 2, 2),\\
&h_3 = (0, 0, 1, 0, 1, 2, 0, 1 , 2, 3, 0, 1, 2, 3, 3),
&h_4 = (0, 0, 1, 0, 1, 2, 0, 1, 2, 3, 0, 1, 2, 3, 4), \\
&h_5 =(0, 1, 1, 1, 1, 1, 1, 1, 1, 1, 1, 1,1 , 1, 1),
&h_6 =(0, 1, 1, 1, 2, 2, 1, 2, 2, 2, 1, 2, 2, 2, 2), \\
&h_7 =(0, 1, 1, 1, 2, 2, 1, 2, 3, 3, 1, 2, 3, 3, 3), 
&h_8 =(0, 1, 1, 1, 2, 2, 1, 2, 3, 3, 1, 2, 3, 4, 4), \\
&h_9 =(0, 1, 1, 2, 2, 2, 2, 2, 2, 2, 2, 2, 2, 2, 2), 
&h_{10} =(0, 1, 1, 2, 2, 2,2, 3, 3, 3, 2, 3, 3, 3, 3), \\
&h_{11} = (0, 1, 1, 2, 2, 2, 2, 3, 3, 3, 2, 3, 4, 4, 4), 
&h_{12} =(0, 1, 1, 2, 2, 2, 3, 3,3 , 3, 3, 3, 3, 3, 3),\\
&h_{13} =(0, 1, 1, 2, 2, 2, 3, 3, 3, 3, 3, 4, 4, 4, 4), 
&h_{14} =(0, 1, 1, 2, 2, 2, 3, 3, 3, 3, 4, 4, 4, 4, 4), \\
&h_{15} =(0, 1, 2, 2, 3, 3, 2, 3, 4, 4,2, 3, 4, 4, 4), 
&h_{16} =(0, 1, 2, 2, 3, 3, 2, 3, 4, 4, 2, 3, 4, 5, 5), \\
&h_{17} =(0, 1, 2, 2, 3, 4, 2, 4, 5, 5, 2, 4, 5, 6, 6), 
&h_{18} =(0, 1, 2, 2, 3, 3, 3, 4, 4, 4, 3, 4, 5, 5, 5), \\
&h_{19} =(0, 1, 2, 2, 3, 4, 3, 4, 5, 5, 3, 4, 5, 6, 6), 
&h_{20} =(0, 2, 2, 3, 4, 4, 4, 5, 5,5, 4, 5, 6, 6, 6).
\end{align*}

Now, by computing, with for example Macaulay 2 \cite{GS}, the kernel of the map from
the polynomial algebra $\mathbb{C}[x_1,...,x_{20}]$ to the hive algebra $\mathrm{HA}(4)$
sending $x_i$ to $\mathbf{z}^{h_i}$, we obtain the following $15$
basic relations:
\begin{align*}
&r'_1: h_1 + h_7 +  h_9 = h_6 + h_{15}, 
&r'_2: h_1 + h_8 + h_9 = h_6 + h_{16}, \\
&r'_3: h_1 + h_{11} + h_{12} = h_{10} + h_{18}, 
&r'_4: h_6 + h_{11} + h_{12} = h_{10} + h_{20}, \\
&r'_5: h_2 + h_8 + h_{10} = h_7 + h_{17}, 
&r'_6: h_2 + h_8 + h_{12} = h_7 + h_{19}, \\
&r'_7: h_6 + h_{18} = h_1 + h_{20}, 
&r'_8: h_7 + h_{16} = h_8 + h_{15}, \\
&r'_9: h_{10} + h_{19} = h_{12} + h_{17},  
&r'_{10}: h_{15} + h_{17} = h_{2} + h_{10}  + h_{16}, \\
&r'_{11}: h_{15} + h_{19} = h_{2} + h_{12} + h_{16}, 
&r'_{12}: h_{15} + h_{20} = h_{7} + h_{9} + h_{18}, \\
&r'_{13}: h_{16} + h_{20} = h_8 + h_{9} + h_{18}, 
&r'_{14}: h_{17} + h_{20} = h_{6} + h_{11} + h_{19},\\
&r'_{15}: h_{17} + h_{18} = h_{1} + h_{11} + h_{19}.
\end{align*}

With \eqref{degree-d}, the degrees of the generators are 
given by $h_{55}$, and therefore we have
\begin{align*}
& \deg(\mathbf{z}^{h_1})=1, \ \deg(\mathbf{z}^{h_2})=2, \ \deg(\mathbf{z}^{h_3})=3, 
\ \deg(\mathbf{z}^{h_4})=4, \ \deg(\mathbf{z}^{h_5})=1, \\
& \deg(\mathbf{z}^{h_6})=2, \ \deg(\mathbf{z}^{h_7})=3, \ \deg(\mathbf{z}^{h_8})=4, 
\ \deg(\mathbf{z}^{h_9})=2, \ \deg(\mathbf{z}^{h_{10}})=3, \\
& \deg(\mathbf{z}^{h_{11}})=4, \ \deg(\mathbf{z}^{h_{12}})=3, \ 
\deg(\mathbf{z}^{h_{13}})=4, \ \deg(\mathbf{z}^{h_{14}})=4, \ 
\deg(\mathbf{z}^{h_{15}})=4, \\
& \deg(\mathbf{z}^{h_{16}})=5, \ \deg(\mathbf{z}^{h_{17}})=6, \ 
\deg(\mathbf{z}^{h_{18}})=5, \ \deg(\mathbf{z}^{h_{19}})=6, \ 
\deg(\mathbf{z}^{h_{20}})=6.
\end{align*}
Now by using the software Normaliz \cite{BIRS}, we can compute 
the HP series of the hive cone $\mathcal{H}(4)$,
which is $f(t)/g(t)$ given in the statement. See \S \ref{appendix} for more details.

\bigskip

%%%%%%%%%%%%%%%%%%%%%%%%%%%%%%%%%%%%%%%%%%%%%%%%%%%%%%%%%%%%%%%%%%%%%%%%%%%%%%
%%%%%%%%%%%%%%%%%%%%%%%%%%%%%%%%%%%%%%%%%%%%%%%%%%%%%%%%%%%%%%%%%%%%%%%%%%%%%%

\section{Generators of Highest Weight Vectors}\label{tensor-algebra}

%%%%%%%%%%%%%%%%%%%%%%%%%%%%%%%%%%%%%%%%%%%%%%%%%%%%%%%%%%%%%%%%%%%%%%%%%%%%%%

In this section, we first review the construction of 
the $\mathrm{GL}(n)$ tensor product algebra and its properties  
given in \cite{HTW05, HJLTW}. Then, for each $n=2, 3, 4$, we give 
a finite presentation of the algebra $\mathrm{TA}(n)$ 
in terms of generators and relations. 
This will give a method to construct  
highest weight vectors appearing in the decomposition of 
the tensor product of two irreducible polynomial 
representations of $\mathrm{GL}(n)$.

%%%%%%%%%%%%%%%%%%%%%%%%%%%%%%%%%%%%%%%%%%%%%%%%%%%%%%%%%%%%%%%%%%%%%%%%%%%%%%
\subsection{$\mathrm{GL}(n)$ tensor product algebra}

Let us recall the $\mathrm{GL}(n)$ tensor product algebra\footnote{Howe et al. 
investigated a family of algebras parameterized by three positive integers $n$, $p$, and $q$. 
In this paper, we will focus on the case $n=p=q$ related to 
the most general form of the LR coefficients.} introduced  
by Howe et al. in \cite{HTW05, HJLTW}. See also \cite{HL12, Le13}.
We consider two copies of the space $M_n$ of $n\times n$ complex matrices,
and use the coordinates $(x_{ij})$ and $(y_{ij})$ respectively. 
Therefore, a typical element $(m_1, m_2)$ in the space $M_n \oplus M_n$ is
\[
 \left[
\begin{array}{cccccc}
x_{11} &  \cdots & x_{1n} &    y_{11} & \cdots  & y_{1n} \\
\vdots &  \ddots & \vdots  &  \vdots   & \ddots & \vdots    \\
x_{n1} &  \cdots & x_{nn} &    y_{n1} & \cdots  & y_{nn} 
\end{array} \right].
\]
The three copies of $\mathrm{GL}(n)$ act on 
the algebra $\mathbb{C}[M_n \oplus M_n]$ of polynomials on $M_n \oplus M_n$ 
as follows:
\[
( g \cdot f )(m_1, m_2) =  f(g_1^t m_1 g_2, g_1^t m_2 g_3 )
\]
for $g=(g_1, g_2, g_3) \in \mathrm{GL}(n) \times \mathrm{GL}(n) 
\times \mathrm{GL}(n)$ and $f\in \mathbb{C}[M_n \oplus M_n]$.

We write $U_n$ for the maximal unipotent subgroup of $\mathrm{GL}(n)$ consisting 
of all the upper triangular matrices with $1$'s on the diagonal, and
$A_n$ for the maximal torus of $\mathrm{GL}(n)$ consisting of all the
diagonal matrices. 

\begin{definition}
The $\mathrm{GL}(n)$ tensor product algebra $\mathrm{TA}(n)$ is 
the algebra of polynomials in $\mathbb{C}[M_n \oplus M_n]$ 
invariant under the subgroup $U_n \times U_n \times U_n$ 
of $\mathrm{GL}(n) \times \mathrm{GL}(n) \times \mathrm{GL}(n)$
\begin{align*}
\mathrm{TA}(n) &= \mathbb{C}[M_n \oplus M_n]^{U_n \times U_n \times U_n} \\
	&=\{ f \in \mathbb{C}[M_n \oplus M_n] : g \cdot f = f 
	\text{\ for all $g \in U_n \times U_n \times U_n$}\}.
\end{align*}
\end{definition}

Since $A_n$ normalizes $U_n$, 
the space $\mathrm{TA}(n)$ consists of 
weight vectors under the action of $A_n \times A_n \times A_n$.
In \cite{HTW05}, it is shown that, as an algebra, $\mathrm{TA}(n)$ is 
graded by the triples $(\lambda,\mu,\nu)$ of non-negative dominant sequences;
the $(\lambda,\mu,\nu)$-homogeneous component 
$\mathrm{TA}(n)_{\mu \nu}^{\lambda}$ consists of the highest weight vectors
of the isomorphic copies of $V^{\lambda}_n$ 
occurring in the decomposition of the tensor product $V^{\mu}_n \otimes V^{\nu}_n$;
the dimension of the space $\mathrm{TA}(n)_{\mu \nu}^{\lambda}$ is equal to
the LR coefficient $c_{\mu \nu}^{\lambda}$. 
Moreover, for each space $\mathrm{TA}(n)_{\mu \nu}^{\lambda}$, 
they constructed explicit $\mathbb{C}$-basis elements $f_T$  associated with 
LR tableaux $T$ on $\lambda/\mu$ with content $\nu$.
\begin{lemma} \cite{HTW05, HJLTW} \label{initial-m}
\begin{enumerate}
\item There is a $\mathbb{C}$-basis for the space $\mathrm{TA}(n)_{\mu \nu}^{\lambda}$ 
\begin{equation*}
 \mathcal{B}_n (\lambda,\mu, \nu)
  = \left\{ f_T \in \mathrm{TA}(n)  : 
   \text{LR tableaux $T$ on $\lambda/\mu$ with content $\nu$} \right\}
\end{equation*}
such that (with respect to a certain monomial order) 
the initial monomial $in(f_T)$  of $f_T$ is
\begin{equation*}
in(f_T) = \prod_{i} x_{ii}^{\mu_i} \cdot \prod_{i,j} y_{ij}^{t_{ij}}
 \in \mathbb{C}[M_n \oplus M_n]
\end{equation*}
where $t_{ij}$ is the number of boxes in the $i$th row of $T$ containing $j$.
In particular, the initial monomials of these basis elements are distinct.
\item The initial algebra $in(\mathrm{TA}(n))$ of $\mathrm{TA}(n)$ is 
generated by $in(f_T)$ for all the LR tableaux $T$ for $\mathrm{GL}(n)$.
\begin{align*}
in(\mathrm{TA}(n)) &= \{in(f) : f \in \mathrm{TA}(n) \} \\
  &= \{ in(f_T) : \text{for all the LR tableaux $T$ for $\mathrm{GL}(n)$} \}.
\end{align*}
\end{enumerate}
\end{lemma}

From this, it is also shown that 
there is a flat one parameter family of algebras 
whose general fiber is $\mathrm{TA}(n)$ and special fiber is  
$in(\mathrm{TA}(n))$. See \cite{HTW05, HJLTW} for more details.

\begin{proposition}\label{in(TA)-HA-bij}
The initial algebra $in(\mathrm{TA}(n))$ of $\mathrm{TA}(n)$ is isomorphic 
to the hive algebra $\mathrm{HA}(n)$.
\end{proposition}
\begin{proof}
One can easily show that the bijection
between hives $h$ and LR tableaux $T$ (Lemma \ref{LR-Hive-bij}) and 
the bijection between LR tableaux $T$ and initial monomials $in(f_T)$ of basis elements $f_T$ 
(Lemma \ref{initial-m}) give rise to a monoid isomorphism 
between $\mathcal{H}(n)$ and $in(\mathrm{TA}(n))$.
Alternatively, one can use the similar result given in terms of 
LR triangles in \cite[\S2.4]{HJLTW} combined with 
a bijection between hives and LR triangles \cite{PV05}. 
\end{proof}

%%%%%%%%%%%%%%%%%%%%%%%%%%%%%%%%%%%%%%%%%%%%%%%%%%%%%%%%%%%%%%%%%%%%%%%%%%%%%%
\subsection{Presentation of $\mathrm{TA}(n)$}
%%%%%%%%%%%%%%%%%%%%%%%%%%%%%%%%%%%%%%%%%%%%%%%%%%%%%%%%%%%%%%%%%%%%%%%%%%%%%%

Now,  we investigate a finite presentation of the $\mathrm{GL}(n)$ tensor 
product algebra, thereby giving a way to express highest weight vectors 
with generators of the algebra $\mathrm{TA}(n)$. 

The following notation will be useful to describe our generators of the algebra
$\mathrm{TA}(n)$. It will also make it easier to notice the relation between 
the products of generators and the LR tableaux corresponding to them.

\begin{notation}\label{tab-notation}
The determinant of the $(i+k)\times(i+k)$ submatrix of
\[
\left[
\begin{array}{cccccc}
x_{11} &  \cdots & x_{1n} &    y_{11} & \cdots  & y_{1n} \\
\vdots &  \ddots & \vdots  &  \vdots   & \ddots & \vdots    \\
x_{n1} &  \cdots & x_{nn} &    y_{n1} & \cdots  & y_{nn} 
\end{array} \right]
\]
obtained by choosing rows $1,2,...,i+k$ and columns 
$1,2,...,i, n+j_1, n+j_2, ..., n+j_k$ will be denoted by 
a script-sized Young tableau with a single column having $i$ empty 
boxes followed by boxes with 
entries $j_\ell$'s for $1 \leq j_1 < j_2 < ... < j_k \leq n$.
Then, its initial monomial is $\prod_{a=1}^{i} x_{aa} \prod_{b=1}^{k} y_{b{j_b}}$ 
with respect to a diagonal term order (see \cite{MS05}).
\end{notation}

With this notation, the product of column tableaux is
the product of the corresponding determinants. For example,
\[
{\scriptsize {\young(\ ,p,q) \cdot \young(r,s) \cdot \young(\ ) }}
=
\det \begin{bmatrix}
x_{11} &  y_{1p} &  y_{1q} \\
x_{21} &  y_{2p} &  y_{2q} \\
x_{31} &  y_{3p} &  y_{3q} 
\end{bmatrix}
\times \det 
\begin{bmatrix}
 y_{1r} &  y_{1s} \\
 y_{2r} &  y_{2s}
\end{bmatrix}
\times 
\det \begin{bmatrix} x_{11} \end{bmatrix}
\]
and its initial monomial is $x^2_{11} y_{2p} y_{3p} y_{1r}y_{2s}$.

%%%%%%%%%%%%%%%%%%%%%%%%%%%%%%%%%%%%%%%%
\subsubsection{Tensor product algebra for $\mathrm{GL}(2)$} 
Now we give a finite presentation of the tensor product algebra for  $\mathrm{GL}(2)$.

\begin{theorem}\label{gl2-thm}
The $\mathrm{GL}(2)$ tensor product algebra $\mathrm{TA}(2)$ is generated by
\begin{equation*}
\Yvcentermath1
g_1 = {\scriptsize {\young(1,2)}},\ \  g_2= {\scriptsize {\young(\ ,1)}}, \ \ g_3={\scriptsize {\young(\ )}},
 \ \ g_4={\scriptsize {\young(\ ,\ )}}, \ \ g_5={\scriptsize {\young(1)}}
\end{equation*}
and these generators are algebraically independent. Therefore, 
\[
\mathrm{TA}(2) \cong \mathbb{C}[z_1, z_2, ..., z_5].
\]
\end{theorem}

\begin{proof}
We first note that all these determinants $g_i$ are 
invariant under $U_n \times U_n \times U_n$ and therefore they belong to 
the tensor product algebra $\mathrm{TA}(2)$. We also note that they are 
algebraically independent.
Computing their weights under $A_n \times A_n \times A_n$, we see that 
$g_i$ is a highest weight vector of $V_2^{\lambda}$ 
in the decomposition of $V_2^{\mu} \otimes V_2^{\nu}$ where
\begin{center}
\begin{tabular}{ c | c | c | c }
    $i$ & $\lambda$ & $\mu$   & $\nu$    \\ \hline 
    $1$ & $(1,1)$   & $(0,0)$ & $(1,1)$  \\ 
    $2$ & $(1,1)$   & $(1,0)$ & $(1,0)$  \\ 
    $3$ & $(1,0)$   & $(1,0)$ & $(0,0)$  \\
    $4$ & $(1,1)$   & $(1,1)$ & $(0,0)$  \\
    $5$ & $(1,0)$   & $(0,0)$ & $(1,0)$  \\ 
\end{tabular}
\end{center}

Because these tensor products are multiplicity free, our highest weight 
vector of $V_2^{\lambda}$ in the decomposition of 
$V_2^{\mu} \otimes V_2^{\nu}$ is unique up to constant. 
We note that the column tableaux labeling the determinants $g_i$
are LR tableaux for $\mathrm{GL}(2)$.
In fact, one can check that $g_i$ is equal to $f_{T_i}$ given in 
Lemma \ref{initial-m} where $T_i$ is the tableau labeling $g_i$ by 
Notation \ref{tab-notation}.
Moreover, these LR tableaux $T_i$ correspond to 
the Hilbert basis elements $h_i$ listed in the proof of 
Theorem \ref{hive-alg-gl2} by Lemma \ref{LR-Hive-bij}.
Since $h_i$ generate the hive cone $\mathcal{H}(2)$, 
by Proposition \ref{in(TA)-HA-bij}, the monomials $in(g_i)$ generate 
the initial algebra of $\mathrm{TA}(2)$.
Therefore $g_i$ generate the algebra $\mathrm{TA}(2)$.
\end{proof}

We remark that, with the tableau notation of determinants 
(Notation \ref{tab-notation}), 
every LR tableau for $\mathrm{GL}(2)$ can be matched with 
a concatenation of the tableaux for $g_i$.
For example,  the LR tableaux listed in \eqref{LR-tab-gl2} can 
be matched with the concatenation of the tableaux labeling 
$g_i$, or the following products of $g_i$'s:
\begin{align*}
&g_3^3,  \ \ g_3 g_4,  \ \ g_3^2 g_5,  \ \ g_2 g_3,  \ \  g_4 g_5,\\
&g_3 g_5^2,  \ \ g_2 g_5,  \ \  g_1 g_3,  \ \  g_5^3,  \ \  g_1 g_5.
\end{align*}
Here we remark that 
$g_1 g_3 = {\scriptsize {\young(1,2) \cdot \young(\ )}}$ can be rewritten as
\begin{equation}\label{straightening-1}
{\scriptsize {\young(1,2) \cdot \young(\ )}}
 = {\scriptsize {\young(\ ,2) \cdot \young(1) - \young(\ ,1)\cdot\young(2)}},
\end{equation}
by using the polynomial identity 
\begin{equation}\label{straightening-111Q}
\det 
\begin{bmatrix}
 y_{11} &  y_{12} \\
 y_{22} &  y_{22}
\end{bmatrix}
\cdot 
x_{11}
=
\det 
\begin{bmatrix}
 x_{11} &  y_{12} \\
 x_{21} &  y_{22}
\end{bmatrix}
\cdot 
y_{11}
-
\det 
\begin{bmatrix}
 x_{11} &  y_{11} \\
 x_{21} &  y_{21}
\end{bmatrix}
\cdot 
y_{12}.
\end{equation}
Note that in \eqref{straightening-1}, we obtain a LR tableau by 
aligning the columns 
${\scriptsize {\young(\ ,2)}}$ and ${\scriptsize {\young(1)}}$ to the top,
and this LR tableau is the one matched with the product $g_1 g_3$.

%%%%%%%%%%%%%%%%%%%%%%%%%%%%%%%%%%%%%%%%
\subsubsection{Tensor product algebra for $\mathrm{GL}(3)$}

Next we describe the $\mathrm{GL}(3)$ tensor product algebra 
in terms of generators and relations.

\begin{theorem}\label{gl3-thm}
The $\mathrm{GL}(3)$ tensor product algebra $\mathrm{TA}(3)$ is generated by 
\begin{align*}
& g_1 = {\scriptsize {\young(1)}}, \ \ \  g_2= {\scriptsize {\young(1,2,3)}}, \ \ \  g_3={\scriptsize {\young(1,2)}}, \ \ \  
  g_4= {\scriptsize {\young(\ )}}, \ \ \ g_5={\scriptsize {\young(\ ,1)}}, \\  
& g_6={\scriptsize {\young(\ ,1,2)}}, \ \ \ g_7= {\scriptsize {\young(\ ,\ )}}, \ \ \  g_8={\scriptsize {\young(\ ,\ ,1)}},\ \ \ 
 g_9={\scriptsize {\young(\ ,\ ,\ )}}, \ \ \  g_{10}={\scriptsize {\young(\ ,\ ,2)\cdot 
  \young(1) -\young(\ ,\ ,1)\cdot \young(2)}},
\end{align*}
and there is one basic relation 
$g_1 g_6 g_7 - g_5 g_{10} + g_3 g_4 g_8 = 0$ among them. Therefore, 
\[
\mathrm{TA}(3) \cong \mathbb{C}[z_1, z_2, ..., z_{10}]/ 
\langle z_1 z_6 z_7 - z_5 z_{10} + z_3 z_4 z_8 \rangle.
\]
\end{theorem}
\begin{proof}
We note that all these polynomials $g_i$ are 
 invariant under the action of $U_n \times U_n \times U_n$, and therefore 
 they belong to the tensor product algebra $\mathrm{TA}(3)$.
Computing their weights under $A_n \times A_n \times A_n$, we see that  
$g_i$ is a highest weight vector of $V_3^{\lambda}$ in the decomposition of 
$V_3^{\mu} \otimes V_3^{\nu}$ where
\begin{center}
\begin{tabular}{ c | c | c | c }
    $i$ & $\lambda$ & $\mu$   & $\nu$    \\ \hline 
    $1$ & $(1,0,0)$   & $(0,0,0)$ & $(1,0,0)$  \\ 
    $2$ & $(1,1,1)$   & $(0,0,0)$ & $(1,1,1)$   \\ 
    $3$ & $(1,1,0)$   & $(0,0,0)$ & $(1,1,0)$  \\ 
    $4$ & $(1,0,0)$   & $(1,0,0)$ & $(0,0,0)$  \\ 
    $5$ & $(1,1,0)$   & $(1,0,0)$ & $(1,0,0)$  \\ 
    $6$ & $(1,1,1)$   & $(1,0,0)$ & $(1,1,0)$  \\ 
    $7$ & $(1,1,0)$   & $(1,1,0)$ & $(0,0,0)$  \\ 
    $8$ & $(1,1,1)$   & $(1,1,0)$ & $(1,0,0)$  \\ 
    $9$ & $(1,1,1)$   & $(1,1,1)$ & $(0,0,0)$  \\ 
    $10$ & $(2,1,1)$   & $(1,1,0)$ & $(1,1,0)$  \\ 
\end{tabular}
\end{center}

In each case, since  $V_3^{\mu}$ or $V_3^{\nu}$ is a fundamental representation 
of $\mathrm{GL}(3)$, the tensor product $V_3^{\mu} \otimes V_3^{\nu}$ is 
multiplicity free. Hence, our highest weight vectors are unique up to constant. 
Let $T_i$ be the column tableau labeling the determinant $g_i$ for $1\leq i \leq 9$ 
by Notation \ref{tab-notation}.
For $g_{10}$, we let $T_{10}$ be the concatenation of two columns in 
the first term of $g_{10}$. Then, all these tableaux $T_i$ are LR tableaux, and 
$g_i$ is indeed equal to $f_{T_i}$ given in Lemma \ref{initial-m}. 
Moreover, these LR tableaux,  by Lemma \ref{LR-Hive-bij}, 
correspond to the Hilbert basis elements $h_i$ listed in the proof of 
Theorem \ref{Hive-alg-GL3}.
Since these $h_i$ generate the hive cone 
$\mathcal{H}(3)$, by Proposition \ref{in(TA)-HA-bij}, the monomials $in(g_i)$ generate 
the initial algebra of $\mathrm{TA}(3)$.
Therefore, the polynomials $g_i$ generate the algebra $\mathrm{TA}(3)$.

Now, for the basic relation between these generators, 
by lifting the relation between hives obtained in Theorem \ref{Hive-alg-GL3}, 
we can find the polynomial identity in the statement.
\end{proof}

We remark that using the above result   
every LR tableau for $\mathrm{GL}(3)$ can be matched with a concatenation of the LR tableaux 
$T_i$ associated with $g_i$, or simply the products of $g_i$'s. For example,  the LR tableaux for 
$\mathrm{GL}(3)$ listed in \eqref{LR-tab-gl3} 
can be matched with
\begin{align*}
&g_4^3,  \ \ g_4 g_7,  \ \ g_9,  \ \ g_1 g_4^2,  \ \  g_4 g_5, \ \ g_1 g_7, \ \ g_8,\\
&g_1^2 g_4,  \ \ g_1 g_5,  \ \  g_3 g_4,  \ \  g_6, \ \ g_1^3, \ \ g_1 g_3, \ \ g_2.
\end{align*}
Here, as in \eqref{straightening-1}, the product $g_3 g_4 = {\scriptsize {\young(1,2) 
\cdot \young(\ )}}$ is equal to 
${\scriptsize {\young(\ ,2) \cdot \young(1) - \young(\ ,1)\cdot\young(2)}}$, 
and we see that the first term in the new expression indeed matches with the LR tableau $T_1$ below  
once the columns ${\scriptsize {\young(\ ,2)}}$ and ${\scriptsize {\young(1)}}$ are top-aligned.
\[
T_1 = \young(\ 1,2) \quad \text{and} \quad T_2 = \young(\ \ ,\ 1,2)
\] 
Similarly, the product $g_6 g_7 = {\scriptsize {\young(\ ,1,2) \cdot \young(\ ,\ )}}$  is equal to 
${\scriptsize {\young(\ ,\ ,2)\cdot \young(\ ,1) - \young(\ ,\ ,1)\cdot \young(\ ,2)}}$ 
by the corresponding polynomial identity, and 
therefore $g_6 g_7$ can be matched with the LR tableau $T_2$ 
when the columns ${\scriptsize {\young(\ ,\ ,2)}}$ and ${\scriptsize {\young(\ ,1)}}$ 
in the new expression of  $g_6 g_7$ are top-aligned.

%%%%%%%%%%%%%%%%%%%%%%%%%%%%%%%%%%%%%%%%
\subsubsection{Example} \label{ex33}

For $\mathrm{GL}(3)$, if $\mu=\nu=(2,1,0)$ and $\lambda=(3,2,1)$, 
then there are exactly two copies of $V^{\lambda}_3$  
occurring in the decomposition of the tensor product 
$V^{\mu}_3 \otimes  V^{\nu}_3$. 
That is, the LR coefficient $c^{\lambda}_{\mu \nu}$ is $2$.

This multiplicity can be computed by counting 
all the LR tableaux on 
skew diagram $\lambda/\mu$ with content $\nu$. They are
\begin{equation*}
T =\young(\ \ 1,\ 2,1) \text{\ \ and \ \ }
T' =\young(\ \ 1,\ 1,2).
\end{equation*}
Alternatively, we can count 
all the hives with boundary $(\lambda, \mu, \nu)$.  
\begin{equation*}
h=\left[
\begin{array}{ccccccc}
   &    &    & 0      &    &   &         \\
   &    & 2  &        & 3  &   &        \\
   & 3  &    & 4 &    & 5 &        \\
 3 &    & 5  &        & 6  &   & 6 
\end{array}
\right]  \text{\ \ and \ \ }
h'=
\left[
\begin{array}{ccccccc}
   &    &    & 0      &    &   &         \\
   &    & 2  &        & 3  &   &        \\
   & 3  &    & 5 &    & 5 &        \\
 3 &    & 5  &        & 6  &   & 6 
\end{array}
\right].
\end{equation*}

To find explicit highest weight vectors of the isomorphic copies of 
$V^{\lambda}_3$ corresponding to $h$ and $h'$ (or $T$ and $T'$) using our
generators $g_1,...,g_{10}$ of $\mathrm{TA}(3)$, we first express the hives 
$h$ and $h'$ with our Hilbert basis of the cone $\mathcal{H}(3)$
\begin{equation*}
h=h_3 + h_4 + h_8  \text{\ \ and \ \ } h'=h_1 + h_6 + h_7,
\end{equation*}
which implies 
\begin{equation*}
\mathbf{z}^h=\mathbf{z}^{h_3} \cdot \mathbf{z}^{h_4} \cdot \mathbf{z}^{h_8}
\text{\ \ and \ \ }
\mathbf{z}^{h'}=\mathbf{z}^{h_1} \cdot \mathbf{z}^{h_6} \cdot \mathbf{z}^{h_7}
\end{equation*}
in the hive algebra $\mathrm{HA}(3) \cong in(\mathrm{TA}(n))$. 
By lifting these expressions to the tensor product algebra $\mathrm{TA}(4)$, 
we  obtain the highest weight vectors corresponding 
to $h$ and $h'$ respectively.
\begin{align} \label{hwv-f1}
f & = g_3 \cdot g_4 \cdot g_8
 = {\scriptsize {\young(1,2) \cdot \young(\ ) \cdot \young(\ ,\ ,1)}}  \\
 &= det\left[
\begin{array}{ccccc}
 y_{11} & y_{12} \\
 y_{21} & y_{22}
\end{array} \right]
\times
x_{11}
\times
det\left[
\begin{array}{ccccc}
x_{11} &  x_{12} & y_{11}\\
x_{21} &  x_{22} & y_{21} \\
x_{31} &  x_{32} & y_{31}
\end{array} \right]  \notag\\
&  = {\scriptsize {\young(\ ,\ ,1)\cdot \young(\ ,2) \cdot \young(1) - \young(\ ,\ ,1) 
 \cdot \young(\ ,1) \cdot \young(2)}}. \notag
\end{align}
\begin{align}\label{hwv-f2} 
f' &= g_1 \cdot g_6 \cdot g_7 = 
{\scriptsize {\young(1) \cdot \young(\ ,1,2) \cdot \young(\ ,\ )}} \\
  &= y_{11} \times
det\left[
\begin{array}{ccccc}
x_{11} &  y_{11} & y_{12} \\
x_{21} &  y_{21} & y_{22} \\
x_{31} &  x_{31} & y_{32} 
\end{array} \right]
\times
det\left[
\begin{array}{ccccc}
x_{11} &  x_{12} \\
x_{21} &  x_{22} 
\end{array} \right]  \notag \\
& = {\scriptsize {\young(\ ,\ ,2)\cdot \young(\ ,1) \cdot \young(1) - \young(\ ,\ ,1) 
 \cdot \young(\ ,2) \cdot \young(1)}}. \notag
\end{align}

We note that these $f$ and $f'$ are indeed elements in the $(\lambda,\mu,\nu)$-homogeneous 
component $\mathrm{TA}(3)_{\mu\nu}^{\lambda}$. 
With the tableau notation for determinants, 
by top-aligning the columns  ${\scriptsize {\young(\ ,\ ,1)}}$, 
${\scriptsize {\young(\ ,2)}}$, and ${\scriptsize {\young(1)}}$ 
in the first term of the last expression of $f$ in \eqref{hwv-f1},
we can match $f$ with the LR tableau $T$.
Also, by top-aligning the columns ${\scriptsize {\young(\ ,\ ,2)}}$, 
${\scriptsize {\young(\ ,1)}}$, 
and ${\scriptsize {\young(1)}}$ in the first term of 
the last expression of $f'$ in \eqref{hwv-f2}, 
we can match $f'$ with the LR tableau $T'$. 
Note that one can check that $f$ and $f'$ have distinct initial monomials 
\begin{equation*}
in(f)=x_{11}^2 x_{22} y_{11} y_{22} y_{31}\text{\ \ and\ \ } 
in(f')= x_{11}^2 x_{22} y_{11} y_{21} y_{32} 
\end{equation*} 
with respect to the monomial order given in \cite{HTW05}, and 
these initial monomials can be
matched with the LR tableaux $T$ and $T'$ as in Lemma \ref{initial-m}.
 
Finally, we remark that the highest weight vectors 
corresponding to $T$ and $T'$ by the formula in \cite{HTW05} 
are different from our highest weight vectors $f$ and $f'$. 
In fact, they are $f$ and $-(f + f')$. 
See the computations in  \cite[\S 8.6]{Le13}.

%%%%%%%%%%%%%%%%%%%%%%%%%%%%%%%%%%%%%%%%
\subsubsection{Tensor product algebra for $\mathrm{GL}(4)$}

Next we investigate the $\mathrm{GL}(4)$ tensor product algebra.

\begin{theorem}\label{gl4-thm}
The $\mathrm{GL}(4)$ tensor product algebra $\mathrm{TA}(4)$ is generated by 
the following twenty polynomials 
\begin{align*}
& g_1 = {\scriptsize {\young(1)}}, \ \ \   g_2={\scriptsize {\young(1,2)}}, \ \ \   g_3={\scriptsize {\young(1,2,3)}}, 
\ \ \   g_4= {\scriptsize {\young(1,2,3,4)}}, \ \ \   g_5= {\scriptsize {\young(\ )}}, \ \ \  
 g_6= {\scriptsize {\young(\ ,1)}}, \ \ \   g_{7}={\scriptsize {\young(\ ,1,2)}}, \\
 & v_{8}={\scriptsize {\young(\ ,1,2,3)}}, 
\ \ \   g_9= {\scriptsize {\young(\ ,\ )}}, \ \ \  g_{10}={\scriptsize {\young(\ ,\ ,1)}}, \ \ \  
 g_{11}={\scriptsize {\young(\ ,\ ,1,2)}}, \ \ \   g_{12}= {\scriptsize {\young(\ ,\ ,\ )}}, 
\ \ \  g_{13}={\scriptsize {\young(\ ,\ ,\ ,1)}}, \ \ \  g_{14}={\scriptsize {\young(\ ,\ ,\ ,\ )}}, \\  
 & g_{15}= {\scriptsize {\young(\ ,\ ,2) \cdot \young(1) -  \young(\ ,\ ,1) \cdot \young(2)}},
\ \ \   
 g_{16}= {\scriptsize {\young(\ ,\ ,2,3)\cdot \young(1) - \young(\ ,\ ,1,3) \cdot \young(2) 
 + \young(\ ,\ ,1,2) \cdot \young(3)}}, \ \ \    
  g_{17}= {\scriptsize {\young(\ ,\ ,1,3) \cdot \young(1,2) - \young(\ ,\ ,1,2) \cdot 
 \young(1,3)}}, \\
& g_{18}= {\scriptsize {\young(\ ,\ ,\ ,2) \cdot \young(1) - \young(\ ,\ ,\ ,1) \cdot \young(2)}},  \ \ \  
 g_{19}= {\scriptsize {\young(\ ,\ ,\ ,3) \cdot \young(1,2)  - \young(\ ,\ ,\ ,2) 
 \cdot \young(1,3) +  \young(\ ,\ ,\ ,1) \cdot \young(2,3)}}, \ \ \  
 g_{20}= {\scriptsize {\young(\ ,\ ,\ ,2) \cdot \young(\ ,1) - \young(\ ,\ ,\ ,1) 
 \cdot \young(\ ,2)}}
\end{align*}
with the fifteen basic relations $\hat{r}_j=0$ among them where 
\begin{align*}
&\hat{r}_1 = g_1 g_7 g_9 - g_6 g_{15} + g_2 g_5 g_{10},
&\hat{r}_2 = g_1 g_8 g_9 - g_6 g_{16}  + g_5 g_{17}, \\
&\hat{r}_3 = g_1 g_{11} g_{12} - g_{10} g_{18} + g_{13} g_{15}, 
&\hat{r}_4 = g_6 g_{11} g_{12} - g_{10} g_{20} + g_7 g_9 g_{13}, \\
&\hat{r}_5 = g_2 g_8 g_{10} - g_7 g_{17} + g_3 g_6 g_{11},  
&\hat{r}_6 = g_2 g_8 g_{12} - g_7 g_{19} + g_3 g_{20}, \\
&\hat{r}_7 = g_6 g_{18} - g_1 g_{20} + g_2 g_5 g_{13}, 
&\hat{r}_8 = g_7 g_{16} - g_8 g_{15} - g_3 g_5 g_{11}, \\
&\hat{r}_9 = g_{10} g_{19} - g_{12} g_{17} - g_3 g_9 g_{13}, 
&\hat{r}_{10} = g_{15} g_{17} - g_{2} g_{10} g_{16} - g_1 g_3 g_9 g_{11}, \\
&\hat{r}_{11} = g_{15} g_{19} - g_{2} g_{12} g_{16} - g_3 g_9 g_{18}, 
&\hat{r}_{12} = g_{15} g_{20} - g_{7} g_{9} g_{18} - g_2 g_5 g_{11} g_{12}, \\
&\hat{r}_{13} = g_{16} g_{20} - g_{5} g_{11} g_{19} - g_8 g_{9} g_{18}, 
&\hat{r}_{14} = g_{17} g_{20} - g_{6} g_{11} g_{19} - g_2 g_8 g_{9} g_{13},\\
&\hat{r}_{15} = g_{17} g_{18} - g_{1} g_{11} g_{19} - g_{2} g_{13} g_{16}. &
\end{align*}
\end{theorem}
We remark that, as in the previous cases, with our tableau notation of 
determinants, the first terms of our generators can be, once their columns 
are top-aligned, identified with LR tableaux.
We first note that all these polynomials $g_i$ are 
invariant under the action of $U_n \times U_n \times U_n$, and therefore 
belong to the tensor product algebra $\mathrm{TA}(4)$.
Computing their weights under $A_n \times A_n \times A_n$, we see that 
they are indeed highest weight vectors appearing in the decomposition of 
tensor products: $g_i$ is a highest weight vector of $V_4^{\lambda}$ 
in the decomposition of $V_4^{\mu} \otimes V_4^{\nu}$ where
\begin{center}
\begin{tabular}{ c | c | c | c }
    $i$ & $\lambda$ & $\mu$   & $\nu$    \\ \hline 
    $1$ & $(1,0,0,0)$   & $(0,0,0,0)$ & $(1,0,0,0)$  \\ 
    $2$ & $(1,1,0,0)$   & $(0,0,0,0)$ & $(1,1,0,0)$  \\ 
    $3$ & $(1,1,1,0)$   & $(0,0,0,0)$ & $(1,1,1,0)$  \\ 
    $4$ & $(1,1,1,1)$   & $(0,0,0,0)$ & $(1,1,1,1)$  \\ 
    $5$ & $(1,0,0,0)$   & $(1,0,0,0)$ & $(0,0,0,0)$  \\ 
    $6$ & $(1,1,0,0)$   & $(1,0,0,0)$ & $(1,0,0,0)$  \\ 
    $7$ & $(1,1,1,0)$   & $(1,0,0,0)$ & $(1,1,0,0)$  \\ 
    $8$ & $(1,1,1,1)$   & $(1,0,0,0)$ & $(1,1,1,0)$  \\ 
    $9$ & $(1,1,0,0)$   & $(1,1,0,0)$ & $(0,0,0,0)$  \\ 
    $10$ & $(1,1,1,0)$   & $(1,1,0,0)$ & $(1,0,0,0)$  \\ 
    $11$ & $(1,1,1,1)$   & $(1,1,0,0)$ & $(1,1,0,0)$  \\ 
    $12$ & $(1,1,1,0)$   & $(1,1,1,0)$ & $(0,0,0,0)$  \\ 
    $13$ & $(1,1,1,1)$   & $(1,1,1,0)$ & $(1,0,0,0)$  \\ 
    $14$ & $(1,1,1,1)$   & $(1,1,1,1)$ & $(0,0,0,0)$  \\ 
    $15$ & $(2,1,1,0)$   & $(1,1,0,0)$ & $(1,1,0,0)$  \\ 
    $16$ & $(2,1,1,1)$   & $(1,1,0,0)$ & $(1,1,1,0)$  \\ 
    $17$ & $(2,2,1,1)$   & $(1,1,0,0)$ & $(2,1,1,0)$  \\ 
    $18$ & $(2,1,1,1)$   & $(1,1,1,0)$ & $(1,1,0,0)$  \\ 
    $19$ & $(2,2,1,1)$   & $(1,1,1,0)$ & $(1,1,1,0)$  \\ 
    $20$ & $(2,2,1,1)$   & $(2,1,1,0)$ & $(1,1,0,0)$  \\ 
\end{tabular}
\end{center}

We note that $V_4^{\mu}$ or $V_4^{\nu}$ is a fundamental representation of 
$\mathrm{GL}(4)$, and therefore $V_4^{\mu} \otimes V_4^{\nu}$ is multiplicity-free.
Therefore, the highest weight vector of $V_4^{\lambda}$ in the decomposition of 
$V_4^{\mu} \otimes V_4^{\nu}$ is unique. Now with the same arguments we used 
for $\mathrm{GL}(2)$ and $\mathrm{GL}(3)$, by lifting the computations done for $\mathcal{H}(4) \cong
in(\mathrm{TA}(4))$ to $\mathrm{TA}(4)$, we see that $in(g_i)$ generate the initial algebra of
$\mathrm{TA}(4)$, and therefore $g_i$ generate $\mathrm{TA}(4)$.
The relations between $g_j$ can be also obtained easily by lifting 
the relations of hives computed in Theorem \ref{Hive-alg-4}.

\subsection{Remarks on the generators and relations}
The algebra of invariants of the standard maximal unipotent subgroup 
$U_p$ of $\mathrm{GL}(p)$ under left  multiplication on the space $M_{p,q}$ of 
$p \times q$ matrices is generated by certain minors over $M_{p,q}$. 
This is a well-known result from classical invariant theory. 
See, for example, \cite[\S 9]{Fu97} and \cite[\S 14]{MS05}. 
Hence it is not so surprising that the elements of the $\mathrm{GL}(n)$ 
tensor product algebra  $\mathrm{TA}(n)$ 
can be expressed in terms of certain minors over $M_{n, 2n} \cong M_n \oplus M_n$ and 
their relations can be realized in terms of classical determinantal 
identities such as the quadratic relations of Sylvester type given in \cite[\S 8]{Fu97}. 

Here we give some interesting examples of such identities. Recall 
that the Lewis Carroll identity \cite{Dod}, 
also known as the Desnanot-Jacobi identity, for a $n \times n$ matrix $A$ is 
\[
\det A \det A' \; = \; \det A_{n}^{n} \det A_{1}^{1} - \det A_{n}^{1} \det A_{1}^{n}
\]
where $A'$ is the $(n-2) \times (n-2)$ submatrix of $A$ obtained by erasing the first 
and last rows of $A$ and the first and last columns of $A$; $A_p^q$ is 
the $(n-1)\times (n-1)$ submatrix of $A$ obtained by erasing the $p$th row 
and $q$th column of $A$. Then, 
the relations $\hat{r}_1=0$, $\hat{r}_3=0$, and $\hat{r}_5=0$ for $\mathrm{TA}(4)$  
in Theorem \ref{gl4-thm} are essentially the Lewis Carroll identity applied to 
the following matrices 
\[
\begin{bmatrix}
0 & 0 & 1 & 0 \\
y_{11} & x_{11} & x_{12} & y_{12} \\
y_{21} & x_{21} & x_{22} & y_{22} \\
y_{31} & x_{31} & x_{32} & y_{32}

\end{bmatrix},  \;
\begin{bmatrix}
0 & 0 & 0 & 1 & 0 \\
y_{11} & x_{11} & x_{12} & x_{13} & y_{12} \\
y_{21} & x_{21} & x_{22} & x_{23} & y_{22} \\
y_{31} & x_{31} & x_{32} & x_{33} & y_{32} \\
y_{41} & x_{41} & x_{42} & x_{43} & y_{42} 

\end{bmatrix}, \;  \text{and}  \; 
\begin{bmatrix}
0 & 0 & 0 & 1 & 0 \\
y_{12} & x_{11} & y_{11} & x_{12} & y_{13} \\
y_{22} & x_{21} & y_{21} & x_{22} & y_{23}  \\
y_{32} & x_{31} & y_{31} & x_{32} & y_{33} \\
y_{42} & x_{41} & y_{41} & x_{42} & y_{43} 
\end{bmatrix}
\]
respectively. We refer the interested reader to \cite[\S 3]{K18} 
and \cite[\S 5]{KY17} for similar observations.

\bigskip

%%%%%%%%%%%%%%%%%%%%%%%%%%%%%%%%%%%%%%%%%%%%%%%%%%%%%%%%%%%%%%%
\section{Appendix: Normaliz  codes for $\mathcal{H}(4)$}\label{appendix}
%%%%%%%%%%%%%%%%%%%%%%%%%%%%%%%%%%%%%%%%%%%%%%%%%%%%%%%%%%%%%%%

In this section, we provide the input codes for jNormaliz (a Java-based 
graphical interface for Normaliz) we used to analyze the hive cone 
$\mathcal{H}(4)$ together with their outputs. For the interpretation of 
these outputs other than the ones mentioned in Theorem \ref{Hive-alg-4}, 
we refer the readers to Normaliz documents \cite{BIRS}.

\subsection{Hilbert basis} 
We first define the hive cone $\mathcal{H}(4)$ using inequalities.

\begin{lstlisting}
30
15
-1 1 0 0 0 0 0 0 0 0 0 0 0 0 0  /*the boundary is nonnegative dominant*/
0 -1 0 1 0 0 0 0 0 0 0 0 0 0 0
0 0 0 -1 0 0 1 0 0 0 0 0 0 0 0
0 0 0 0 0 0 -1 0 0 0 1 0 0 0 0
0 0 0 0 0 0 0 0 0 0 -1 1 0 0 0
0 0 0 0 0 0 0 0 0 0 0 -1 1 0 0
0 0 0 0 0 0 0 0 0 0 0 0 -1 1 0
0 0 0 0 0 0 0 0 0 0 0 0 0 -1 1
-1 0 1 0 0 0 0 0 0 0 0 0 0 0 0
0 0 -1 0 0 1 0 0 0 0 0 0 0 0 0 
0 0 0 0 0 -1 0 0 0 1 0 0 0 0 0
0 0 0 0 0 0 0 0 0 -1 0 0 0 0 1
-1 1 1 0 -1 0 0 0 0 0 0 0 0 0 0  /* the rhombus conditions */
0 -1 0 1 1 0 0 -1 0 0 0 0 0 0 0
0 0 -1 0 1 1 0 0 -1 0 0 0 0 0 0
0 0 0 -1 0 0 1 1 0 0 0 -1 0 0 0
0 0 0 0 -1 0 0 1 1 0 0 0 -1 0 0
0 0 0 0 0 -1 0 0 1 1 0 0 0 -1 0
0 -1 1 0 1 -1 0 0 0 0 0 0 0 0 0
0 0 0 -1 1 0 0 1 -1 0 0 0 0 0 0
0 0 0 0 -1 1 0 0 1 -1 0 0 0 0 0
0 0 0 0 0 0 -1 1 0 0 0 1 -1 0 0 
0 0 0 0 0 0 0 -1 1 0 0 0 1 -1 0 
0 0 0 0 0 0 0 0 -1 1 0 0 0 1 -1 
0 1 -1 -1 1 0 0 0 0 0 0 0 0 0 0
0 0 0 1 -1 0 -1 1 0 0 0 0 0 0 0
0 0 0 0 1 -1 0 -1 1 0 0 0 0 0 0
0 0 0 0 0 0 1 -1 0 0 -1 1 0 0 0
0 0 0 0 0 0 0 1 -1 0 0 -1 1 0 0 
0 0 0 0 0 0 0 0 1 -1 0 0 -1 1 0
inequalities

1
15
1 0 0 0 0 0 0 0 0 0 0 0 0 0 0    /* h_11 = 0 */
equations
\end{lstlisting}

\smallskip

Its output gives a Hilbert basis for $\mathcal{H}(4)$ among others.

\begin{lstlisting}
20 Hilbert basis elements
20 extreme rays
20 support hyperplanes

embedding dimension = 15
rank = 14
external index = 1

size of partial triangulation   = 0
resulting sum of |det|s = 0

No implicit grading found

rank of class group = 6
class group is free

***********************************************************************

20 Hilbert basis elements:
 0 0 1 0 1 1 0 1 1 1 0 1 1 1 1
 0 0 1 0 1 2 0 1 2 2 0 1 2 2 2
 0 0 1 0 1 2 0 1 2 3 0 1 2 3 3
 0 0 1 0 1 2 0 1 2 3 0 1 2 3 4
 0 1 1 1 1 1 1 1 1 1 1 1 1 1 1
 0 1 1 1 2 2 1 2 2 2 1 2 2 2 2
 0 1 1 1 2 2 1 2 3 3 1 2 3 3 3
 0 1 1 1 2 2 1 2 3 3 1 2 3 4 4
 0 1 1 2 2 2 2 2 2 2 2 2 2 2 2
 0 1 1 2 2 2 2 3 3 3 2 3 3 3 3
 0 1 1 2 2 2 2 3 3 3 2 3 4 4 4
 0 1 1 2 2 2 3 3 3 3 3 3 3 3 3
 0 1 1 2 2 2 3 3 3 3 3 4 4 4 4
 0 1 1 2 2 2 3 3 3 3 4 4 4 4 4
 0 1 2 2 3 3 2 3 4 4 2 3 4 4 4
 0 1 2 2 3 3 2 3 4 4 2 3 4 5 5
 0 1 2 2 3 3 3 4 4 4 3 4 5 5 5
 0 1 2 2 3 4 2 4 5 5 2 4 5 6 6
 0 1 2 2 3 4 3 4 5 5 3 4 5 6 6
 0 2 2 3 4 4 4 5 5 5 4 5 6 6 6

20 extreme rays:
 0 0 1 0 1 1 0 1 1 1 0 1 1 1 1
 0 0 1 0 1 2 0 1 2 2 0 1 2 2 2
 0 0 1 0 1 2 0 1 2 3 0 1 2 3 3
 0 0 1 0 1 2 0 1 2 3 0 1 2 3 4
 0 1 1 1 1 1 1 1 1 1 1 1 1 1 1
 0 1 1 1 2 2 1 2 2 2 1 2 2 2 2
 0 1 1 1 2 2 1 2 3 3 1 2 3 3 3
 0 1 1 1 2 2 1 2 3 3 1 2 3 4 4
 0 1 1 2 2 2 2 2 2 2 2 2 2 2 2
 0 1 1 2 2 2 2 3 3 3 2 3 3 3 3
 0 1 1 2 2 2 2 3 3 3 2 3 4 4 4
 0 1 1 2 2 2 3 3 3 3 3 3 3 3 3
 0 1 1 2 2 2 3 3 3 3 3 4 4 4 4
 0 1 1 2 2 2 3 3 3 3 4 4 4 4 4
 0 1 2 2 3 3 2 3 4 4 2 3 4 4 4
 0 1 2 2 3 3 2 3 4 4 2 3 4 5 5
 0 1 2 2 3 3 3 4 4 4 3 4 5 5 5
 0 1 2 2 3 4 2 4 5 5 2 4 5 6 6
 0 1 2 2 3 4 3 4 5 5 3 4 5 6 6
 0 2 2 3 4 4 4 5 5 5 4 5 6 6 6

20 support hyperplanes:
 0 -1  0  1  1  0  0 -1  0  0  0  0  0  0  0
 0 -1  1  0  1 -1  0  0  0  0  0  0  0  0  0
 0  0 -1  0  1  1  0  0 -1  0  0  0  0  0  0
 0  0  0 -1  0  0  1  1  0  0  0 -1  0  0  0
 0  0  0 -1  1  0  0  1 -1  0  0  0  0  0  0
 0  0  0  0 -1  0  0  1  1  0  0  0 -1  0  0
 0  0  0  0 -1  1  0  0  1 -1  0  0  0  0  0
 0  0  0  0  0 -1  0  0  1  1  0  0  0 -1  0
 0  0  0  0  0  0 -1  0  0  0  1  0  0  0  0
 0  0  0  0  0  0 -1  1  0  0  0  1 -1  0  0
 0  0  0  0  0  0  0 -1  1  0  0  0  1 -1  0
 0  0  0  0  0  0  0  0 -1  1  0  0  0  1 -1
 0  0  0  0  0  0  0  0  0  0  0  0  0 -1  1
 0  0  0  0  0  0  0  0  1 -1  0  0 -1  1  0
 0  0  0  0  0  0  0  1 -1  0  0 -1  1  0  0
 0  0  0  0  0  0  1 -1  0  0 -1  1  0  0  0
 0  0  0  0  1 -1  0 -1  1  0  0  0  0  0  0
 0  0  0  1 -1  0 -1  1  0  0  0  0  0  0  0
 0  1 -1 -1  1  0  0  0  0  0  0  0  0  0  0
 0  1  1  0 -1  0  0  0  0  0  0  0  0  0  0
\end{lstlisting}

\subsection{HP series} 
To obtain the HP series of $\mathrm{HA}(4)$ and other properties of 
the hive cone $\mathcal{H}(4)$, we redefine the cone using the above 
Hilbert basis elements and specify their degrees.

\begin{lstlisting}
amb_space 15
cone 20
 0 0 1 0 1 1 0 1 1 1 0 1 1 1 1
 0 0 1 0 1 2 0 1 2 2 0 1 2 2 2
 0 0 1 0 1 2 0 1 2 3 0 1 2 3 3
 0 0 1 0 1 2 0 1 2 3 0 1 2 3 4
 0 1 1 1 1 1 1 1 1 1 1 1 1 1 1
 0 1 1 1 2 2 1 2 2 2 1 2 2 2 2
 0 1 1 1 2 2 1 2 3 3 1 2 3 3 3
 0 1 1 1 2 2 1 2 3 3 1 2 3 4 4
 0 1 1 2 2 2 2 2 2 2 2 2 2 2 2
 0 1 1 2 2 2 2 3 3 3 2 3 3 3 3
 0 1 1 2 2 2 2 3 3 3 2 3 4 4 4
 0 1 1 2 2 2 3 3 3 3 3 3 3 3 3
 0 1 1 2 2 2 3 3 3 3 3 4 4 4 4
 0 1 1 2 2 2 3 3 3 3 4 4 4 4 4
 0 1 2 2 3 3 2 3 4 4 2 3 4 4 4
 0 1 2 2 3 3 2 3 4 4 2 3 4 5 5
 0 1 2 2 3 3 3 4 4 4 3 4 5 5 5
 0 1 2 2 3 4 2 4 5 5 2 4 5 6 6
 0 1 2 2 3 4 3 4 5 5 3 4 5 6 6
 0 2 2 3 4 4 4 5 5 5 4 5 6 6 6

grading
 0 0 0 0 0 0 0 0 0 0 0 0 0 0 1

\end{lstlisting}

\medskip

The following is a part of the output obtained by running the previous input file. 
It gives among others the HP series of the monoid algebra of 
the hive cone $\mathcal{H}(4)$.

\begin{lstlisting}
20 extreme rays
20 support hyperplanes

embedding dimension = 15
rank = 14
external index = 1
internal index = 1

size of triangulation   = 16
resulting sum of |det|s = 16

grading:
0 0 0 0 0 0 0 0 0 0 0 0 0 0 1 

degrees of extreme rays:
1: 2  2: 3  3: 4  4: 6  5: 2  6: 3  

multiplicity = 1/147456

Hilbert series:
1 -2 -2 10 -2 -24 22 32 -54 -18 80 -14 -72 34 44 -18 -25
-18 44 34 -72 -14 80 -18 -54 32 22 -24 -2 10 -2 -2 1 
denominator with 14 factors:
1: 4  2: 6  12: 4  

degree of Hilbert Series as rational function = -32

The numerator of the Hilbert Series is symmetric.

Hilbert series with cyclotomic denominator:
1 0 -1 2 2 0 1 0 2 2 -1 0 1 
cyclotomic denominator:
1: 14  2: 10  3: 4  4: 4  6: 2  

\end{lstlisting}

\medskip

%%%%%%%%%%%%%%%%%%%%%%%%%%%%%%%%%%%%%%%%%%%%%%%%%%%%%%%%%%%%%%%%%%%%%%%%%%%%%%
%%%%%%%%%%%%%%%%%%%%%%%%%%%%%%%%%%%%%%%%%%%%%%%%%%%%%%%%%%%%%%%%%%%%%%%%%%%%%%

\medskip
\end{document}